\documentclass[12pt,reqno]{amsart}
\usepackage[top=1in, bottom=1in, left=1in, right=1in]{geometry}

\usepackage[usenames,dvipsnames]{color}
\usepackage{tikz}
\usepackage{ifthen}
\usepackage{tikz-3dplot}
\usepackage{pgfplots}
\usepackage{amsfonts}
\usepackage{amssymb}
\usepackage{bbm}

\usepackage{mathrsfs}

\usepackage{lineno}

\numberwithin{equation}{section}

\newtheorem{theorem}{Theorem}[section]
\newtheorem{proposition}[theorem]{Proposition}
\newtheorem{lemma}[theorem]{Lemma}

\theoremstyle{definition}

\newtheorem{definition}[theorem]{Definition}
\newtheorem{example}[theorem]{Example}

\newtheorem{remark}[theorem]{Remark}

\newcommand\<{\langle}
\newcommand\CC{{\mathbb C}}

\newcommand\RR{{\mathbb R}}
\newcommand\PP{{\mathbb P}}
\def\<{\langle}
\def\>{\rangle}
\def\0{\mathbf{0}}
\def\ka{\kappa}
\def\la{\lambda}

\newcommand\ZZ{{\mathbb Z}}

\newcommand\sA{{\mathscr A}}

\newcommand{\Res}{\operatorname{Res}}
\newcommand{\mRes}{\operatorname{mRes}}
\newcommand{\Span}{\operatorname{span}}
\newcommand{\Hom}{\operatorname{Hom}}

\newcommand\conv{{\rm conv}}

\newcommand\vol{{\rm vol}}
\newcommand\mvol{{\rm mvol}}

\newcommand\minus{\smallsetminus}

\renewcommand\>{\rangle}

\pgfplotsset{compat=1.12}

\begin{document}

\title{Sparse versions of the Cayley--Bacharach Theorem}

\author{Laura Felicia Matusevich}
\address{Department of Mathematics \\
Texas A\&M University \\ College Station, TX 77843.}
\email{laura@math.tamu.edu}
\thanks{LFM was partially supported by NSF Grant DMS--1500832}

\author{Bruce Reznick}
\address{Department of Mathematics \\
University of Illinois at Urbana-Champaign \\ Urbana, IL 61801.}
\email{reznick@math.uiuc.edu}
\thanks{BR was partially supported by Simons Collaboration Grant 280987}

\subjclass[2010]{Primary: 51N35, 52B20. Secondary: 14N10, 13P15}


\begin{abstract}
We give combinatorial generalizations of the Cayley--Bacharach
theorem and induced map.
\end{abstract}
\maketitle

\setcounter{tocdepth}{2}
\parskip=1ex
\parindent0pt

\section{Introduction}
\label{sec:intro}

The \emph{Cayley--Bacharach theorem} states that, given two
cubic curves in the projective plane that meet in nine points, any other
cubic that passes through eight of the points, contains also the
ninth. As with many attributions in Mathematics, it is known that
the Cayley--Bacharach theorem is originally due to Chasles;
the article~\cite{cayley-bacharach} contains a thorough historical account
of this result, including the roles of Cayley and Bacharach, as well
as many geometric generalizations. 

The Cayley--Bacharach theorem provides a map assigning a point in the
plane to eight given points. More precisely, 
given eight points in the affine plane $\CC^2$, the set of cubic polynomials in two variables that
vanish at these points is a two dimensional
vector space, at least if no three of the points are
on a line, and no six lie on a conic. If $F(x,y)$, $G(x,y)$ form a 
basis of this vector space, B\'ezout's Theorem implies that $F$ and
$G$ have nine common zeros (in the projective plane)
counting multiplicity. For most choices of eight points, all the
multiplicities equal $1$, so a ninth point is determined. Again, for
most choices of eight points, this ninth point lies in the affine plane.
This is a natural map from a dense subset of $(\CC^2)^8$ to $\CC^2$.
We call this the \emph{extra point map}, $\Upsilon$.

The map $\Upsilon$ is rational, and explicit formulas for it can be found in
the article~\cite{cbFormulas}. We give a different proof of rationality, as well as
alternative formulas, in Section~\ref{sec:CayleyBacharach}. While our
formulas are less beautiful and more complicated than those
in~\cite{cbFormulas}, and no one in their right mind would use them in
a practical context, they do have one very positive attribute: they
can be naturally generalized.

The generalizations we are interested in are of the following form
\begin{quote}
  A choice of $N$ generic points in $\CC^d$ gives rise to $d$ hypersurfaces satisfying
  given constraints, and it turns out that those hypersurfaces meet in
  exactly $N+1$ points.
\end{quote}
The constraints we consider are combinatorial in nature: we fix the
monomials that appear in the defining equations of the corresponding
hypersurfaces. These support sets must be carefully chosen; in Section~\ref{sec:generalization}, we call
them \emph{Chasles configurations} and \emph{Chasles structures}.
While Chasles' work has received
significant recognition (his name is on the Eiffel tower), we thought
it appropriate to name our generalizations in his honor.
 
Our main result, Theorem~\ref{thm:rationalMap}, states that the extra
point map arising in
this more general situation is still rational. The proof is
essentially the same as our proof that $\Upsilon$ is rational
(see Theorem~\ref{thm:octamapIsRational}), and consequently also produces explicit formulas. 

The key ingredient we use to compute $\Upsilon$ and its generalizations
is the notion of~\emph{resultant}. The well-known resultant of two
univariate polynomials $f$ and $g$ is a polynomial in the coefficients
of $f$ and $g$ that vanishes precisely when $f$ and $g$ have a common
factor. Resultants are a very important tool when solving polynomial
equations, due to their fundamental
role in elimination theory. This has spurred much interest in
resultants, and especially in explicit formulas for resultants. We make 
use of the following fact (that is made precise in Theorem~\ref{thm:ProductOfRoots}):
\begin{quote}
The product of the coordinates of the roots of a system of polynomial equations is a
rational function of the coefficients of the system, that can be expressed in terms of resultants.
\end{quote}
This result has been known since the late 1990s; 
see~\cite{khovanskii,cds,rojas} and
also~\cite{pedersen-sturmfels}. (In this article, we use the version from~\cite{dandrea-sombra}.)
Its relevance is that
it allows us to express the coordinates of the point we are interested
in, in terms of the coordinates of the points we are given and the
coefficients of the polynomials that define the hypersurfaces
containing those points. Those coefficients can also be expressed as
rational functions of the coordinates of the given points, since we
have fixed the monomials that appear in those polynomials.

\subsection*{Outline}

In Section~\ref{sec:CayleyBacharach}, we prove that $\Upsilon$ is
rational by giving an explicit formula in terms of Sylvester
resultants. Section~\ref{sec:Hilbert} explains our motivation for
considering $\Upsilon$ in the first place. Section~\ref{sec:theory}
is a technical section containing results needed to generalize the Cayley--Bacharach
theorem. The paper becomes readable again in
Section~\ref{sec:generalization}, where we introduce our
generalizations, and
Section~\ref{sec:examples} contains infinitely many examples.

\subsection*{Acknowledgments}

This project has benefited from visits by the authors to each other's 
institutions. We also discussed this material during the joint meeting
of the AMS and the RSME in Seville in the Summer of 2003. 
We are very grateful to Bernd Sturmfels,
for his thoughtful advice, and for directing us to the Cayley octads
example in Subsection~\ref{sec:cayley}.
Thanks also to Frank Sottile for productive conversations.

\section{The extra point map is rational}
\label{sec:CayleyBacharach}

We start by showing that the map arising from the Cayley--Bacharach
theorem is rational.

\begin{theorem}
\label{thm:octamapIsRational}
Let $\Upsilon$ be the map that assigns, to eight generic points in the
plane, a ninth point determined by the Cayley--Bacharach theorem.
The map $\Upsilon$ is rational.
\end{theorem}

\begin{proof}
We show that $\Upsilon$ is rational by giving an explicit formula. 
We denote the eight given points by $\rho_i=(a_i,b_i) \in
(\CC^*)^2, 1 \leq i \leq 8$. By the genericity assumption, we may assume that two linearly
independent cubic polynomials vanishing on $\rho_1,\dots,\rho_8$ are
of the form
\[
F(x,y) = x^3 + \ka_{21} x^2y +  \ka_{12} xy^2 + \ka_{20}x^2+\ka_{02}y^2+
               \ka_{11}xy +\ka_{10}x + \ka_{01}y +\ka_{00} 
\]
\[
G(x,y) = y^3 + \la_{21} x^2y +\la_{12} xy^2 + \la_{20}x^2+\la_{02}y^2+
               \la_{11}xy +\la_{10}x + \la_{01}y +\la_{00} 
\]
Let $M$ be the matrix whose rows are 
$[a_i^2b_i,a_ib_i^2,a_i^2,b_i^2,a_ib_i,a_i,b_i,1]$ for
$i=1,\dots,8$. Then 
$[\kappa_{21},\kappa_{12},\kappa_{20},\kappa_{02},\kappa_{11},\kappa_{10}\kappa_{01},\kappa_{00}]^t$
is a solution of 
\[
M v = [-a_1^3,\dots,-a_8^3]^t,
\] 
and 
$[\lambda_{21},\lambda_{12},\lambda_{20},\lambda_{02},\lambda_{11},\lambda_{10}\lambda_{01},\lambda_{00}]^t$
is a solution of 
\[
M v = [-b_1^3,\dots,-b_8^3]^t.
\] 
Again, as
$\rho_1,\dots,\rho_8$ are generic, we may assume that $\det(M)\neq 0$,
and consequently we may explicitly write the coefficients $\kappa_{ij}$ and
$\lambda_{ij}$ as ratios of polynomials in $a_i,b_j$ using Cramer's
rule. For instance
\[
\kappa_{00} = \frac{ \det 
\begin{bmatrix} 
a_1^2b_1 & a_1b_1^2 & a_1^2 & b_1^2 & a_1b_1 & a_1 & b_1 & -a_1^3 \\
a_2^2b_2 & a_2b_2^2 & a_2^2 & b_2^2 & a_2b_2 & a_2 & b_2 & -a_2^3 \\
a_3^2b_3 & a_3b_3^2 & a_3^2 & b_3^2 & a_3b_3 & a_3 & b_3 & -a_3^3 \\
a_4^2b_4 & a_4b_4^2 & a_4^2 & b_4^2 & a_4b_4 & a_4 & b_4 & -a_4^3 \\
a_5^2b_5 & a_5b_5^2 & a_5^2 & b_5^2 & a_5b_5 & a_5 & b_5 & -a_5^3 \\
a_6^2b_6 & a_6b_6^2 & a_6^2 & b_6^2 & a_6b_6 & a_6 & b_6 & -a_6^3 \\
a_7^2b_7 & a_7b_7^2 & a_7^2 & b_7^2 & a_7b_7 & a_7 & b_7 & -a_7^3 \\
a_8^2b_8 & a_8b_8^2 & a_8^2 & b_8^2 & a_8b_8 & a_8 & b_8 & -a_8^3 \\
\end{bmatrix} 
}{\det M}
\]

Now consider $F$ and $G$ as polynomials in the variable $x$, with
coefficients that are polynomials in $y$, and take the resultant to
eliminate the variable $x$. The roots of this resultant (as a
polynomial in $y$) are the $y$-coordinates of the nine solutions of
$F = G = 0$. Consequently, the coefficient of $y^0$ in this resultant
is the product of those nine $y$-coordinates. By taking resultant with
respect to $y$, we can also obtain the product of the nine
$x$-coordinates of the solutions of $F=G=0$.

The above calculation can be performed explicitly using a computer
algebra system. We used Macaulay2~\cite{M2} to compute the
coefficients we are interested in, which are given below.

The zeroth coefficient for the resultant of $F$ and $G$ with respect
to $y$ is:
\begin{align*}
R_x=&
\kappa_{02}^3\lambda_{00}^2
-\kappa_{02}^2\kappa_{01}\lambda_{01}\lambda_{00}
+\kappa_{02}\kappa_{01}^2\lambda_{02}\lambda_{00}
+\kappa_{02}^2\kappa_{00}\lambda_{01}^2
-2\kappa_{02}^2\kappa_{00}\lambda_{02}\lambda_{00}
-\kappa_{02}\kappa_{01}\kappa_{00}\lambda_{02}\lambda_{01} +\\
&
\kappa_{02}\kappa_{00}^2\lambda_{02}^2
-\kappa_{01}^3\lambda_{00}
+3\kappa_{02}\kappa_{01}\kappa_{00}\lambda_{00}
+\kappa_{01}^2\kappa_{00}\lambda_{01}
-2\kappa_{02}\kappa_{00}^2\lambda_{01}
-\kappa_{01}\kappa_{00}^2\lambda_{02}
+\kappa_{00}^3
\end{align*}
The zeroth coefficient for the resultant of $F$ and $G$ with respect
to $x$ (exchange $\lambda$'s and $\kappa$'s) is:
\begin{align*}
R_y= &
\kappa_{00}^2\lambda_{20}^3-\kappa_{10}\kappa_{00}\lambda_{20}^2\lambda_{10}+\kappa_{20}\kappa_{00}\lambda_{20}\lambda_{10}^2+
\kappa_{10}^2\lambda_{20}^2\lambda_{00}- 
2\kappa_{20}\kappa_{00}\lambda_{20}^2\lambda_{00}-\kappa_{20}\kappa_{10}\lambda_{20}\lambda_{10}\lambda_{00}+ \\
& \kappa_{20}^2\lambda_{20}\lambda_{00}^2-\kappa_{00}\lambda_{10}^3+3\kappa_{00}\lambda_{20}\lambda_{10}\lambda_{00}+ 
\kappa_{10}\lambda_{10}^2\lambda_{00}-2\kappa_{10}\lambda_{20}\lambda_{00}^2-\kappa_{20}\lambda_{10}\lambda_{00}^2+
\lambda_{00}^3
\end{align*}

Then our ninth point, $\rho=(a,b) = \Upsilon(\rho_1,\dots,\rho_8)$,
can be expressed as 
\[
a = \frac{R_x}{a_1\cdots a_8} , \qquad b = \frac{R_y}{b_1\cdots b_8}.
\]
\end{proof}

While the above formulas involve much division and high degree
polynomials, their most significant feature, as has been mentioned
before, is that they can be generalized. But this must wait until Section~\ref{sec:generalization}.

\section{Hilbert}
\label{sec:Hilbert}

This version of the Cayley-Bacharach Theorem was central to Hilbert's 1888 proof 
that there exist positive semidefinite (psd) ternary sextics which cannot be written as a sum of
squares (sos) of real ternary cubics (see~\cite{hilbert1888}). Hilbert starts with two real cubic polynomials
$F(x,y)$ and $G(x,y)$ which have nine common zeros -- $\{\rho_j \mid 1 \leq j \leq 9\}$, no three on a line, 
no six on a conic. He shows how to construct a sextic $H$ which is singular at 
$\{\rho_j \mid 1 \leq j \leq 8\}$ but so that $H(\rho_9) > 0$. By looking separately at the neighborhoods of the
$\rho_j$'s and their complement, Hilbert shows that there exists $\lambda > 0$ so that 
$f = F^2 + G^2 + \lambda H$ is psd; observe that $f(\rho_j) = 0$ for $1 \leq j \leq 8$ and $f(\rho_9) > 0$. 
Suppose now that $f = \sum g_k^2$ for cubics $g_k$. Then $g_k(\rho_j)
= 0$ for $1 \leq j \leq 8$, 
and Cayley-Bacharach implies that $g_k(\rho_9) = 0$ for all $g_k$, a
contradiction, which means that $f$ is not sos.
The condition on $F,G$ ensures that they (and $H$) cannot be particularly simple, and no explicit 
example was constructed in the  subsequent eighty years. 

In 1969, R. M. Robinson (see \cite{robinson}), showed that Hilbert's construction works with a 
simple pair $\{F,G\}$ which do not satisfy his restriction. He took $F(x,y) = x^3-x$
and $F(x,y) = y^3-y$, so that the $\rho_j$'s form the $3 \times 3$ grid: $\{-1,0,1\}^2$. He
then shows (in our notation) that $H(x,y) = (1-x^2)(1-y^2)(1-x^2-y^2)$ fulfills the conditions
of Hilbert's construction and that one may even take $\lambda = 1$. The resulting polynomial 
homo\-ge\-ni\-zes to an even symmetric ternary sextic form:
\begin{align*}
R(x,y,z) :&= (x^3-xz^2)^2 + (y^3-yz^2)^2 + (x^2-z^2)(y^2-z^2)(z^2-x^2-y^2) \\
&= x^6 + y^6 + z^6 - (x^4 y^2 +x^2y^4 +x^4z^2+x^2z^4+y^4z^2+ y^2 z^4) + 3 x^2y^2z^2.
\end{align*}
Robinson proves that this form is psd, by writing $(x^2+y^2)R(x,y,z)$ as
a sum of squares. Hilbert's argument shows that $R$ itself is not sos. 
The original eight zeros ho\-mo\-ge\-ni\-ze to $\{(\pm 1, \pm 1,1), (\pm 1, 0, 1), (0, \pm 1,1)\}$
and $R$ itself has two additional zeros ``at infinity'' $(1,\pm 1,0)$.
The article~\cite{concrete} contains further historical discussion on
psd and sos forms.

Choi, Lam and the second author~\cite{CLR} used Robinson's example as a starting point
to analyze all psd even symmetric sextics in $n \geq 3$. The second author~\cite{lostpaper} 
generalized Robinson's example and showed that Hilbert's construction applies, as 
long as no four of the common zeros $\{\rho_j\}$ are on a line, and no seven are on a conic. 
That paper also contains many worked-out examples.

\section{Sparse polynomial systems and sparse resultants}
\label{sec:theory}

In this section, we collect results on sparse systems of polynomial
equations that are necessary for our generalizations of the
Cayley--Bacharach theorem.

A \emph{configuration of lattice points}, or a \emph{configuration} is a 
finite subset $A$ of $\ZZ^d$. 
The \emph{dimension} of $A$, denoted by $\dim(A)$, is the dimension of the
smallest affine subspace of $\RR^d$ containing all points in $A$. 

If $A$ is a configuration, $\conv(A)$
denotes the convex hull of the elements of $A$ in $\RR^d$; $\conv(A)$
is a \emph{convex lattice polytope} (a convex polytope whose vertices have integer
coordinates).

A configuration $A$ is called \emph{saturated} if $A = \conv(A) \cap
\ZZ^d$, that is, if $A$ equals the set of all lattice points in its
convex hull.

If $A\subset \ZZ^d $ is a $d$-dimensional configuration, its \emph{normalized volume}, 
denoted $\vol(A)$,
is the Euclidean volume of $\conv(A)$, renormalized so that the unit
simplex in $\ZZ^d$ has volume one. More explicitly, the $\vol(A)$ is
$d!$ times the Euclidean volume of $\conv(A)$.

A Laurent polynomial is \emph{supported on a configuration
  $A$} if it is of the form 
  $\sum_{u \in A} \lambda_u x^u \in \CC[x_1^{\pm},\dots,x_d^{\pm}]$.

The following result, due to Kouchnirenko~\cite{kouchnirenko},
illustrates the connection between the combinatorics of configurations and systems of
polynomial equations.

\begin{theorem}
\label{thm:Kouchnirenko}
Let $A\subset \ZZ^d$ be a $d$-dimensional configuration, and let
$F_1,\dots,F_d$ be generic Laurent polynomials supported on $A$. Then
the number of common roots of $F_1,\dots,F_d$ (in $(\CC \minus
\{0\})^d := (\CC^*)^d$) is $\vol(A)$.
\end{theorem}

Our next goal is to give the number of solutions to a system of
Laurent polynomial equations, when the supports of the equations are
not necessarily the same. A note on terminology: when we refer to
\emph{sparse systems of equations}, we mean a system of Laurent
polynomial equations whose supports have been fixed. We start by
introducing a generalization of the notion of volume.

\begin{definition}
\label{def:mixedVolume}
If $P\subset \RR^d$ is a convex polytope and $\lambda \in \RR$, let $\lambda P =\{ \lambda
u \mid u \in P\}$. If $P,Q \subset \RR^d$ are polytopes, their
\emph{Minkowski sum} is $P+Q = \{u+v \mid u \in P, v \in Q\}$.
Let $A_1,\dots,A_d$ be configurations, and denote $P_i =
\conv(A_i)$. 
The the \emph{mixed volume} of 
$P_1,\dots, P_d$ is
\[
\mvol(P_1,\dots,P_d):= \frac{1}{d!} 
\sum_{k=1}^d (-1)^{d-k} 
\sum_{1\leq i_1 < \cdots < i_k \leq d}
\vol(P_{i_1}+\cdots+P_{i_k}).
\]
\end{definition}

The following result is known as the Bernstein, Kouchnirenko and
Khovanskii (or BKK) theorem. In this form, it first appeared in ~\cite{bernstein}.

\begin{theorem}
\label{thm:BKK}
Let $A_1,\dots,A_d \subset \ZZ^d$ be configurations, and denote $P_i =
\conv(A_i)$. For $1\leq i \leq d$, let $F_i = \sum_{a \in A_i} \lambda_{i,a} x^a$ be a
Laurent polynomial with support contained in $A_i$. If the
coefficients $\{\lambda_{i,a} \mid a \in A_i, 1\leq i \leq d \}$ are
sufficiently generic, the system of polynomial equations
$F_1(x)=\cdots=F_d(x)=0$ has precisely $\mvol(P_1,\dots,P_d)$
solutions in $(\CC^*)^d$.
\end{theorem}


We now turn to sparse resultants.
While a system of $d$ generic Laurent
polynomial equations in $d$ variables has solutions (see
Theorem~\ref{thm:BKK}), a system of $d+1$ Laurent polynomials in $d$
variables in general does not. The coefficients of the polynomials in
such a system for which solutions exist are determined by a polynomial
called the \emph{resultant}.

As was mentioned in the introduction, resultants can be used to give
an expression for the product (of the coordinates) of the roots of a sparse system.
The earliest versions of this can be found
in~\cite{khovanskii,cds,rojas,pedersen-sturmfels}. In this article we
use the formulas that appear in~\cite{dandrea-sombra}.

Our first task is to introduce resultants in general.

\begin{definition}
\label{def:sparseResultant}
Let $A_0,A_1,\dots,A_d$ be finite subsets of $\ZZ^d$, and let
$\sA=(A_0,\dots,A_d)$. Write $F_i=F_i(\lambda_i,x) = \sum_{u\in A_i} \lambda_{i,u}x^u$
for a Laurent polynomial supported on $A_i$, where
$\lambda_i=(\lambda_{i,u}\mid u \in A_i)$ are variables representing
the coefficients of $F_i$. Set $\lambda=(\lambda_0,\dots,\lambda_d)$.
Let
\[
\Omega_{\sA} = \{ (x,\lambda) \mid F_0(\lambda_0,x)=\cdots =
F_d(\lambda_d,x)=0\} \subset (\CC^*)^d \times \prod_{i=0}^d \PP^{|A_i|-1}.
\]
If the closure of the image of $\Omega_{\sA}$ under the projection 
$(\CC^*)^d \times \prod_{i=0}^d \PP^{|A_i|-1} \to \prod_{i=0}^d
\PP^{|A_i|-1}$ has codimension $1$, then the
\emph{resultant} $\Res_\sA(F_0,\dots,F_d)$ is defined to
be the unique (up to sign) irreducible polynomial in $\ZZ[\lambda]$ which vanishes
on this hypersurface. If this closure has codimension at least $2$,
then we define $\Res_\sA(F_0,\dots,F_d)=1$.
\end{definition}

\begin{example}
\label{ex:changeLattice}
In our definitions, we have used the lattice $\ZZ^d$ as the ambient
lattice without remarking upon it. In general, we may be given
configurations that naturally live in lattices other than $\ZZ^d$, in
which case, we need to change the way we compute resultants accordingly.

For instance, consider $A=\{(0,0,0),(0,1,0),(-1,0,2)\} \subset \ZZ^3$. Then
$A$ generates a lattice $L$ which is isomorphic to $\ZZ^2$. 
If $F_i = \lambda_{i,1} +\lambda_{i,2} y + \lambda_{i,3}x^{-1}z^2$ are
generic Laurent polynomials supported on $A$ for $i=1,2,3$, the system
$F_0(x,y,z)=F_1(x,y,z)=F_2(x,y,z)=0$ has solutions in $(\CC^*)^3$.
On the other hand, we can also consider this as a system of $3$ equations in the two
variables $(s,t)$ induced by the lattice $L$. This new system is
$\lambda_{i,1} +\lambda_{i,2} s + \lambda_{i,3}t = 0$ for
$i=1,2,3$, whose solvability depends on the vanishing of the
determinant of the $3\times 3$ matrix $[\lambda_{i,j}]$.
Using the convention that the ambient lattice is $L$, this determinant
is the resultant of our system.
\end{example}

In order to use the formulas in~\cite{dandrea-sombra}, what is needed
is a power $\Res_\sA(F_0,\dots,F_d)^{\mu_{\sA}}$. 
The exponent $\mu_{\sA}$ can be given a geometric or combinatorial
definition; 
we use the combinatorial definition
of $\mu_\sA$ that can be found in~\cite[Section~2.1]{esterov}.

We need a preliminary result.
If $J \subset \{0,\dots,d\}$, set $A_J = \cup_{j \in J} A_j$. The
following is~\cite[Corollary~1.1]{sturmfelsNewtonResultant}.

\begin{proposition}
\label{prop:essential}
With the notation of Definition~\ref{def:sparseResultant},
$\Res_\sA(F_0,\dots,F_d)\neq 1$ if and only if there exists a unique subset
$J_0 \subset\{0,\dots,d\}$ such that 
\begin{enumerate}
\item $\dim(\conv(A_{J_0})) - |J_0| = -1$, and
\item for $J \subsetneq J_0$, $\dim(\conv(A_{J})) - |J| > -1$.
\end{enumerate}
\end{proposition}

\begin{definition}
\label{def:powerForResultant}
Let $\sA$ as in Definition~\ref{def:sparseResultant}. Assuming that
$\Res_\sA(F_0,\dots,F_d)\neq 1$, let $J_0 \subseteq \{0,\dots,d\}$ be as in
Proposition~\ref{prop:essential}, and consider $A_{J_0} \times \{1\}
\subset \ZZ^{d+1}$. Let 
\[
\pi: \RR^{d+1} \to \RR^{d+1}/\Span_\RR(A_{J_0}
\times \{1\})
\]
be the projection, and let $\eta$ be a volume form 
on $\RR^{d+1}/\Span_\RR(A_{J_0} \times \{1\})$ such that the volume of 
$\RR^{d+1}/\big(\Span_\RR(A_{J_0} \times \{1\})+\ZZ^{d+1}\big)$ with
respect to $\eta$ equals $(d+1-|J_0|)!$. We set
\begin{equation}
\label{eqn:powerForResultant}
\mu_\sA := 
    \left| \frac{\Span_\RR(A_{J_0} \times \{1\}) \cap
    \ZZ^{d+1}}{\Span_\ZZ(A_{J_0} \times \{1\})} \right| 
\cdot 
    \mvol_\eta\bigg(\{ \conv(\pi(A_i\times\{1\})) \mid i \in \{0,\dots,d \}
    \smallsetminus J_0\}\bigg)
\end{equation}
where $\mvol_\eta$ denotes the mixed volume with respect to the volume
form $\eta$.
\end{definition}

There is one case where $\mu_\sA$ is easy to compute.

\begin{lemma}
\label{lemma:d=1}
Use the notation of Definitions~\ref{def:sparseResultant} and~\ref{def:powerForResultant}.
If each of the sets $A_0,\dots,A_d$ spans $\ZZ^d$ as a lattice, then $\mu_\sA=1$.
\end{lemma}

\begin{proof}
In this case, $J_0=\{0,\dots,d\}$. We have also
$\Span_\RR(A_{J_0} \times \{1\}) \cap
    \ZZ^{d+1}  = \Span_\ZZ(A_{J_0} \times \{1\})$, so that the first
    factor in~\eqref{eqn:powerForResultant} equals $1$, and the second factor
    does not appear, since $\{0,\dots,d\} \smallsetminus J_0=\varnothing$.
\end{proof}

\begin{definition}
Using the notation from Definitions~\ref{def:sparseResultant}
and~\ref{def:powerForResultant}, we set
\begin{equation}
\label{eqn:resultantWithPower}
\mRes_\sA(F_0,\dots,F_d) := \Res_\sA(F_0,\dots,F_d)^{\mu_\sA}.
\end{equation}
\end{definition}

\begin{remark}
\label{remark:notation}
In~\cite{esterov,dandrea-sombra}, the word resultant, and its
corresponding notation, is used for the polynomial
$\mRes_\sA(\lambda)$. In this article, we follow the usual convention
that resultants are irreducible polynomials.
\end{remark}

Our goal is to state a formula from~\cite{dandrea-sombra} that gives
the product of the coordinates of the solutions of a sparse system of
equations in terms of resultants. To do this, we need to introduce
\emph{directional resultants}.

Let $A_1,\dots,A_d$ be finite subsets of $\ZZ^d$, and let
$F_1,\dots,F_d$ be Laurent polynomials such that the support of $F_i
=\sum_{u\in A_i} \lambda_{i,u}x^u$
is (contained in) $A_i$.
Let $v\in \Hom(\ZZ^d,\ZZ) \cong \ZZ^d$. The \emph{weight} of $u \in
\ZZ^d$ with respect to $v$, denoted $v\cdot u$, is the image of $u$
under $v$. 
For $1\leq i \leq d$, let
$A_{i,v}$ be the set of elements of $A_i$ with minimal weight with
respect to $v$, and let $F_{i,v} = \sum_{u \in A_{i,v}}\lambda_{i,u}x^u$
be the restriction of $F_i$ to $A_{i,v}$. Note that $A_{i,v}$ is
contained in a translate of the lattice $v^\perp = \{ u \in \ZZ^d \mid v \cdot u =
0 \} \cong \ZZ^{d-1}$. Choose $\beta_{1,v},\dots,\beta_{d,v}  \in \ZZ^d$ such
that $A_{i,v} - \beta_{i,v} := \{u-\beta_{i,v} \mid u \in A_{i,v} \} \subset
v^\perp$, and let $G_{i,v}  =\sum_{u\in A_i} \lambda_{i,u}x^{u-\beta_{i,v}}$.
We denote

\begin{equation}
\label{eqn:directionalResultant}
\mRes_{\sA,v}(F_{1,v},\dots,F_{d,v}) :=
\mRes_{\{A_{1,v},\dots,A_{d,v}\}}(G_{1,v},\dots,G_{d,v}) 
\end{equation}

In the expression above, the resultant on the right hand side is constructed with respect to
the ambient lattice $v^\perp$ as in Example~\ref{ex:changeLattice}.
We have constructed a polynomial in the coefficients of the $F_{i,v}$,
which is called a \emph{directional resultant}. We note that it is
independent of the choice of $\beta_{1,v},\dots,\beta_{d,v}$.

We are now ready to state the main result of this section, which is a
special case of~\cite[Corollary~1.3]{dandrea-sombra}.

\begin{theorem}
\label{thm:ProductOfRoots}
Let $A_1,\dots,A_d \subset \ZZ^d$ be configurations, and let
$F_1,\dots,F_d$ be Laurent polynomials supported on
$A_1,\dots,A_d$. Denote by $\mathscr{Z}(F_1,\dots,F_d)$ the set of solutions of
$F_1(x)=\cdots=F_d(x)=0$ in $(\CC^*)^d$. If $\rho \in
\mathscr{Z}(F_1,\dots,F_d)$, let $m_\rho$ be its multiplicity as a solution of 
$F_1(x)=\cdots=F_d(x)=0$.
Assume that for all $v \neq 0$, $v \in
\Hom(\ZZ^d,\ZZ)$, we have $\mRes_{\sA,v}(F_{1,v},\dots,F_{d,v})\neq
0$. For $1\leq i \leq d$, we have that
\begin{equation}
\label{eqn:ProductOfRoots}
\prod_{\rho \in \mathscr{Z}(F_1,\dots,F_d)} \rho_i^{m_\rho} = 
\pm \prod \mRes_{\sA,v}(F_{1,v},\dots,F_{d,v})^{v\cdot e_i},
\end{equation}
where the product on the right is over the primitive vectors in $\Hom(\ZZ^d,\ZZ)$,
and $e_1,\dots,e_d$ are the standard unit vectors in $\ZZ^d$.
\end{theorem}

We note that by Proposition~\ref{prop:essential}, if $v$ is not an
inner normal of a codimension $1$ face of
$\conv(A_1)+\cdots+\conv(A_d)$, then 
$\mRes_{\sA,v}(F_{1,v},\dots,F_{d,v}) = 1$. This implies that the
product on the right hand side of~\eqref{eqn:ProductOfRoots} has
finitely many factors different from $1$.

The $\pm$ sign in~\eqref{eqn:ProductOfRoots} is necessary, since
resultants are only defined up to sign.

\begin{remark}
The assumption in Theorem~\ref{thm:ProductOfRoots}, that the
directional resultants do not vanish, is a genericity assumption. It
states that we are working with a system of Laurent polynomial
equations such that none of the facial subsystems have a common root.
\end{remark}

\section{Chasles Configurations and Structures}
\label{sec:generalization}

In this section, we give combinatorial generalizations for the
Cayley--Bacharach theorem.

\begin{definition}
\label{def:ChaslesConfiguration}
Let $A\subset \ZZ^d$ be a $d$-dimensional configuration, and
write $|A|$ for the cardinality of $A$. Then $A$ is a 
\emph{Chasles configuration} if $|A| + 1 = \vol(A)+d$. 
A Chasles configuration $A \subset \ZZ^d$ is \emph{saturated} if
$A=\conv(A) \cap \ZZ^d$.
\end{definition}

If $A$ is a Chasles configuration, let $N=\vol(A)-1$, so that $|A|=N+d$.
Fix $N$ generic points in $\CC^d$. The Laurent polynomials supported
on $A$ that vanish on these $N$ points form a $d$-dimensional vector
space. If $F_1,\dots,F_d$ is a basis for this vector space, then by
Theorem~\ref{thm:Kouchnirenko}, the number of common zeros of
$F_1,\dots,F_d$ is $N+1$. This determines a map $\Upsilon_A$ from an
open subset of $(\CC^d)^N$ to $\CC^d$.

\begin{definition}
\label{def:ChaslesStructure}
More generally, a \emph{Chasles structure} consists of: two positive
integers $N$ and $d$, a partition $d=k_1+\cdots+k_\ell$, and
configurations $A_1,\dots,A_\ell \subseteq \ZZ^d$ such that
$|A_i|=N+k_i$, and $\mvol(\underbrace{A_1,\dots,A_1}_{k_1\text{
    times}},\underbrace{A_2,\dots,A_2}_{k_2\text{
    times}},\dots,\underbrace{A_\ell,\dots,A_\ell}_{k_\ell \text{
    times}})=N+1$. 
We denote $\sA = \{A_1,\dots,A_\ell\}$. We sometimes
abuse notation and call $\sA$ a Chasles structure.

Note that Chasles configuration is a Chasles structure for 
$N=|A|-\dim(A)$, $d=\dim(A)$, with partition $d=k_1$. 
\end{definition}

We now come to the main result in this article.

\begin{theorem}
\label{thm:rationalMap}
A Chasles structure $\sA$ induces a rational map  $\Upsilon_{\sA} : [(\CC^*)^d]^N \dashrightarrow (\CC^*)^d$
\end{theorem}

\begin{proof}
A Chasles structure is set up so that, if we fix $N$ general points
$\rho_1,\dots,\rho_N \in (\CC^*)^d$,  then for each $i=1,\dots,\ell$, those points 
determine a $k_i$-dimensional vector space of Laurent polynomials supported on
$A_i$ that vanish on them. Picking a basis of each
vector space, we obtain $k_1+\cdots+k_\ell=d$ polynomials
$F_1,\dots,F_d$, whose coefficients can be expressed as rational
functions on the coordinates of $\rho_1,\dots,\rho_N$. 

Since the mixed volume of the corresponding Newton polytopes equals
$N+1$, the Laurent polynomials $F_1,\dots,F_d$ have $N+1$ common zeros
in $(\CC^*)^d$ by Theorem~\ref{thm:BKK}. Let $\rho_{N+1}$ be the point
determined in this way. Note that the genericity assumption on
$\rho_1,\dots,\rho_N$ implies that $\rho_{N+1} \notin \{\rho_1,\dots,\rho_N\}$.

On the other hand, for fixed $1\leq i \leq d$, the product of the
$i$th coordinates of $\rho_1,\dots,\rho_{N+1}$ can be expressed as a
rational function on the coefficients of $F_1,\dots,F_d$ via~\eqref{eqn:ProductOfRoots}.

It follows that the $i$th coordinate of $\rho_{N+1}$ is a rational
function of the coordinates of $\rho_1,\dots,\rho_{N}$.
\end{proof}

We emphasize that, because of the explicit nature
of~\eqref{eqn:ProductOfRoots}, the proof of
Theorem~\ref{thm:rationalMap} can be used to provide 
an explicit expression for the map $\Upsilon_\sA$. This is illustrated
in examples in the following section.

\section{Examples}
\label{sec:examples}

\subsection{Example: The Cayley--Bacharach Theorem}

In this case $A$ is the set of lattice points in the triangle in $\RR^2$ with
vertices $(0,0), (3,0),(0,3)$ (affine or inhomogeneous version) or the set
of lattice points in the triangle in $\RR^3$ with vertices $(3,0,0),
(0,3,0), (0,0,3)$. In either case, $A$ consists of $10$ points,
$\dim(A)=2$ and $\vol(A)=9$, so $A$ is a Chasles configuration.

\subsection{Example: Triangle with one interior point}

Let $A=\{(0,0),(1,1),(2,1),(1,2)\}$. This is a saturated Chasles
configuration, with $\vol(A)=3$. Given $2$ generic points $\rho_1=(a_1,b_1)$
and $\rho_2=(a_2,b_2)$, which determine a two-dimensional space of
polynomials supported on $A$ that vanish on $\rho_1,\rho_2$. Denote $(a_3,b_3)$
their third common zero.
We pick the
following basis of this vector space:
\begin{align*}
F & =
a_1a_2b_1b_2(a_1b_2-a_2b_1)+(a_1b_1^2-a_2b_2^2)x^2y+(a_2^2b_2-a_1^2b_1)xy^2, \\ 
G & = 
(b_1-b_2)x^2y+(a_2-a_1)xy^2+(a_1b_2-a_2b_1)xy.
\end{align*}
In this case, the directional resultants are $2\times 2$ determinants
of the coefficients of $F$ and $G$ corresponding to the facets of
$\conv(A)$. The formula~\eqref{eqn:ProductOfRoots} yields
\[
a_1a_2a_3 = \pm
\frac{a_1a_2b_1b_2(a_1-a_2)^2}{(b_1-b_2)(a_1b_1-a_2b_2)}, \qquad
b_1b_2b_3 = \pm
\frac{a_1a_2b_1b_2(b_1-b_2)^2}{(a_1-a_2)(a_1b_1-a_2b_2)}.
\]
Checking the signs, we obtain that
\[
a_3 = - \frac{b_1b_2(a_1-a_2)^2}{(b_1-b_2)(a_1b_1-a_2b_2)}, \qquad
b_3 = - \frac{a_1a_2(b_1-b_2)^2}{(a_1-a_2)(a_1b_1-a_2b_2)}.
\]

Note that $(a_i,b_i)$ $i=1,2,3$ are collinear. To see this note that,
since we are working over $(\CC^*)^2$, the system $F=G=0$ is
equivalent to the sytem $F=\frac{1}{xy}G=0$, and $\frac{1}{xy}G=0$ is
the equation of the line through $(a_1,b_1)$ and $(a_2,b_2)$.

\subsection{Saturated Planar Chasles configurations}

In this section, we show that there are only finitely many isomorphism
classes of saturated
Chasles configurations of dimension two. We note that there are
infinitely $2$-dimensional Chasles structures
involving two different saturated configurations (see Section~\ref{subsec:infChaslesStructure}).

Recall that a lattice polytope is \emph{reflexive} if its polar polytope is
also a lattice polytope.
A lattice polygon is reflexive if and only if
it contains a unique interior lattice point, but this is not
sufficient in higher dimensions. It follows from~\cite{scott,hensley}, 
that the number of equivalence classes (up to translations and
$\textrm{GL}_n(\mathbb{Z})$) of reflexive polytopes is finite. 
In the case of dimension $2$, the number of equivalence classes is
well known to be sixteen; there is an algorithm for computing all such
equivalence classes~\cite{ks1}, which yields $4,319$ equivalence classes
in dimension three~\cite{ks2}, and $473,800,776$ equivalence classes in dimension four~\cite{ks3}.

\begin{proposition}
The saturated Chasles configurations of dimension $2$ correspond to
reflexive polygons. Consequently there are only sixteen isomorphism
classes of saturated Chasles configurations of dimension $2$.
\end{proposition}

\begin{proof}

If $\dim(A)=2$, let ${\rm Int}(A)$ denote the number of lattice points
in the interior of $\conv(A)$, and ${\rm Bd}(A)$ the number of
lattice points on the boundary of $\conv(A)$.

By Pick's formula, the normalized volume of $A$ equals $\vol(A)=2\ {\rm
  Int}(A) + {\rm Bd}(A) - 2$. Combined with the Chasles condition $|A|
+ 1 = \vol(A)+\dim(A)$, we see that ${\rm Int}(A) = 1$.
\end{proof}

\subsection{Example: Cayley Octads}
\label{sec:cayley}

For $d=3$ we consider 
\[
A = \{(0,0,0), (1,0,0), (0,1,0), (0,0,1), (2,0,0), (0,2,0), (0,0,2), 
(1,1,0), (1,0,1), (0,1,1) \},
\] 
the configuration of all lattice points
in twice the standard tetrahedron. Then $A$ is a saturated Chasles
configuration, as $\vol(A)+d = 8+3 = 11 = |A| +1$. 

From a geometric point of view, a polynomial supported on $A$ gives a
quadratic surface, and three of these intersect in eight points
generically. Such configurations of eight points are known
as \emph{Cayley octads}. 

In this case, elegant and compact formulas for the map $\Upsilon_A$ can be found
in~\cite[Proposition~7.1]{octads}. We are grateful to Bernd Sturmfels,
who directed us to this example.

\subsection{Example: A saturated Chasles configuration in dimension
  $d\geq 3$}

For $d\geq 3$ define the configuration 
$A_d = \{ 0, e_1, e_2, e_1+e_2,e_3,2e_3,e_4,\dots,e_d\}\subset \ZZ^d$,
where $e_1,\dots,e_d$ are the coordinate unit vectors.  
The polytope $\conv(A_3)$ is drawn in Figure~\ref{fig:A3}.

Note that the only lattice points in $\conv(A_d)$ belong to $A_d$,
so that $\conv(A_d)$ contains $d+3$ lattice points.
Since $\vol(A_d) = 4$, it follows that $A_d$ is a saturated Chasles
configuration.

\begin{figure}
\begin{center}
\includegraphics[width=.75in]{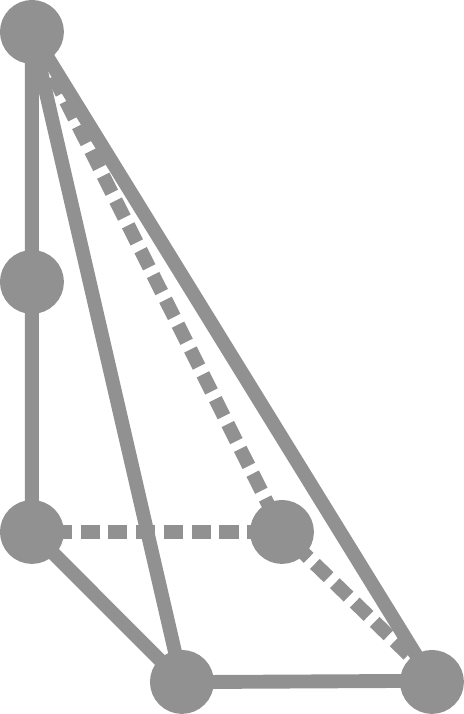}
\end{center}
\caption{The polytope $\conv(A_3)$}
\label{fig:A3}
\end{figure}

The formulas for the extra point map are too large to be displayed directly,
even for $d=3$. For instance, one of the directional resultants
involved is the resultant of three polynomials supported on the unit
square with vertices $\{(0,0),(1,0),(0,1),(1,1)\}$. This is a
polynomial of degree $6$, with $66$ terms, in the $12$ coefficients of the
corresponding polynomials. And this is without taking into account that
those coefficients are themselves rational
functions of the coordinates of the given generic points.

\subsection{Example: Infinite Family of Chasles pairs}
\label{subsec:infChaslesStructure}

Here we produce an infinite family of Chasles structures in the plane.

We let $P_n$ be the quadrangle with vertices $(0,0)$, $(0,n)$, $(1,n+1)$ and
$(1,1)$. We let $Q_n$ be the quadrangle with vertices $(1,0)$, $(0,1)$, $(0,n+1)$ and
$(1,n)$.

$P_n$ and $Q_n$ are reflections of each other, and contain $2n+2$ lattice
points each. Both $P_n$ and $Q_n$ have normalized area $2n$. The Minkowski sum $P_n+Q_n$ is a
hexagon with vertices  
$(1,0)$, $(2,1)$, $(2,2n+1)$, $(1,2n+2)$, $(0,2n+1)$, $(0,1)$, and
normalized area $4(2n+1)$. This is illustrated in Figure~\ref{fig:pairs}.

For the mixed volume, we see that
\[
\mvol(P_n+Q_n) = \frac{1}{2} \big( \vol(P_n+Q_n)-\vol(P_n)-\vol(Q_n)
\big) = \frac{4(2n+1)-2n-2n}{2}  = 2n+2.
\]

The polygons $P_n$ and $Q_n$ thus correspond to a Chasles structure where $d=2$, $k_1=k_2=1$ and $N=2n+1$.
In other words, if we fix $2n+1$ generic points in $\CC^2$, they determine a curve
whose defining polynomial has Newton polytope $P_n$, another curve whose
defining polynomial has Newton polytope $Q_n$, and those two curves meet in $2n+2$
points. 

\begin{figure}
\begin{center}
\includegraphics[width=.365in]{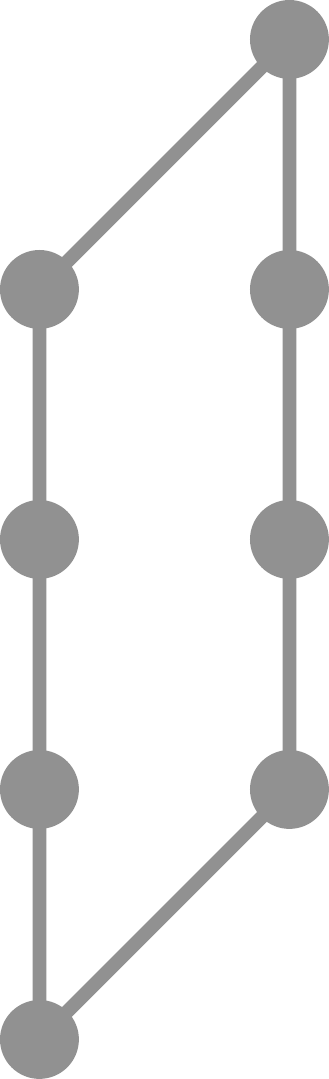}
\hspace{0.3in}
\includegraphics[width=.365in]{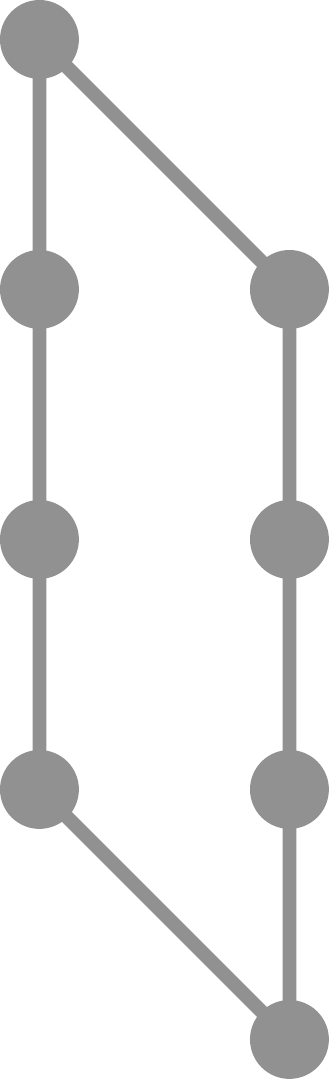}
\hspace{0.3in}
\includegraphics[width=.65in]{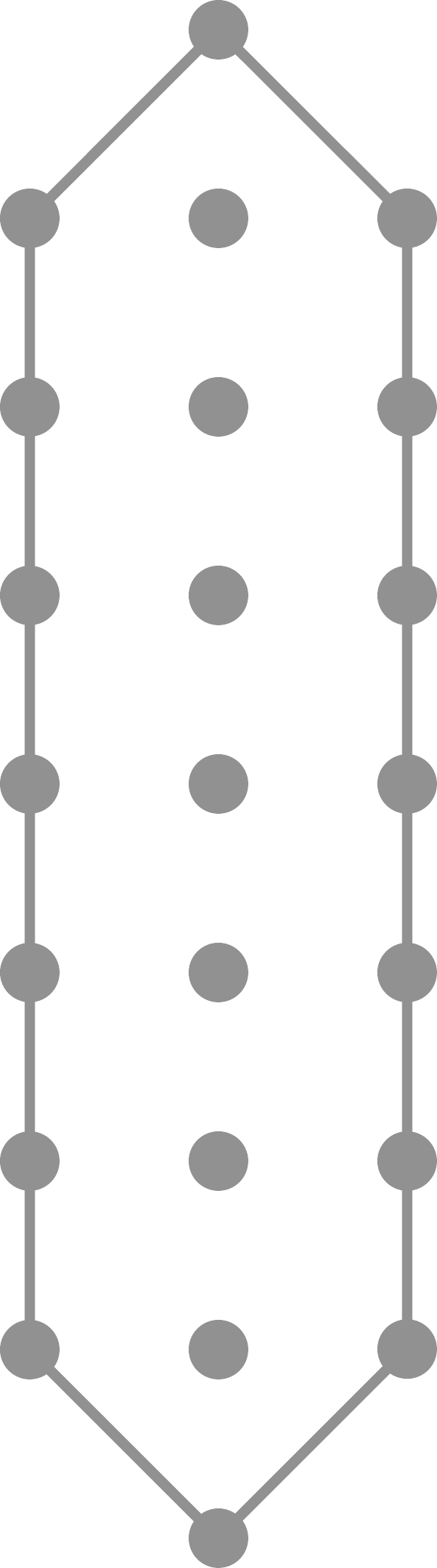}
\end{center}
\caption{The polygons $P_3, Q_3$ and $P_3+Q_3$.}
\label{fig:pairs}
\end{figure}

Since the Minkowski sum $P_n+Q_n$ is a hexagon, there are only six
directional resultants appearing
in~\eqref{eqn:ProductOfRoots}. Choosing the inner normal vectors
$(1,0)$ or $(-1,0)$, the corresponding directional resultant is the classical resultant of two
univariate polynomials of degree $n$. The other four inner normal
vectors yield directional resultants that are monomials.

\subsection{Example: Non-Chasles configuration and non rational extra point map}

Let $A$ be the set of lattice points in the triangle in $\RR^2$ with
vertices $(0,0), (1,2),(3,1)$, with two interior points $(1,1),(2,1)$, so that
$A$ has 5 points, and $\vol(A) = 5$;  $A$ is not a Chasles configuration because three
zeros, in general, induce two more zeros.
Our goal here is to show that the map that assigns the two new zeros
to the original three is not rational. It suffices to fix two of the original zeros and make the third variable,
and show that the coordinates of the new points involve square roots of the coordinates
of the third.

Let us specify that our three points are $(1,1)$, $(2,4)$ and
$(t,t^2)$, $t \neq 0, 1, 2$. There is a two dimensional family of
polynomials supported on $A$ that vanish on these three points. The
following is a choice of basis for this vector space.

\begin{align*}
F_1(x,y) &= x y^2 - x^3 y = x y(y - x^2);\\
F_2(x,y) &= - 8t^3 + (8 + 12 t + 14 t^2 + 15 t^3)x y  \\ &\qquad - (12 + 18 t + 21 t^2 + 7 t^3) x^2 y
+ (4 + 6 t + 7 t^2) x^3 y.
\end{align*}
To find additional common zeros, we note that $F_2(x_0,y_0) \neq 0$ for $x_0y_0 = 0$, and
so any zero also lies on $y_0 = x_0^2$. A computation shows that 
\[
F_2(x_0,x_0^2) = (x_0-1)(x_0-2)(x_0-t) q(x_0); \text{ where }
q(x_0) = (4 + 6t + 7t^2)x_0^2 + (4t + 6t^2)x_0 + 4t^2.
\]
Since the discriminant of $q$ is $-4t^2(12 + 12t + 19t^2)$, which is not a square, the 
values for $x_0$ involves square roots of polynomials involving $t$
and hence  is not a rational function of $\{1,1,2,4,t,t^2\}$.

\raggedbottom
\providecommand{\bysame}{\leavevmode\hbox to3em{\hrulefill}\thinspace}
\providecommand{\MR}{\relax\ifhmode\unskip\space\fi MR }
\providecommand{\MRhref}[2]{%
  \href{http://www.ams.org/mathscinet-getitem?mr=#1}{#2}
}
\providecommand{\href}[2]{#2}

\end{document}